\documentclass{article}
\usepackage{latexsym,epsfig,bm,amssymb}
\usepackage{color}
\usepackage{mathptmx,amsthm,mathrsfs}
\input ot1ptm.fd \input ot1ztmcm.fd \input omlztmcm.fd
\input omsztmcm.fd  \input omxztmcm.fd \input ulasy.fd \input ursfs.fd

\def\version{}

\newcommand{\notyet}[1]{}

\DeclareSymbolFont{AMSb}{U}{msb}{m}{n}
\DeclareSymbolFontAlphabet{\mathbb}{AMSb}

\newcommand\supp{\mathop{\rm supp}}
\newcommand\pa{\partial}

\newcommand\ov{\overline}
\newcommand\om{\omega}

\newcommand{\at}[1]{\vert\sb{\sb{#1}}}

\def\Re{{\rm Re\, }}
\def\Im{{\rm Im\,}}
\providecommand\C{{\mathbb C}}
\renewcommand\C{{\mathbb C}}
\newcommand{\R}{{\mathbb R}}
\newcommand{\N}{{\mathbb N}}

\newcommand{\abs}[1]{\vert #1 \vert}
\newcommand{\Norm}[1]{\left\Vert #1 \right\Vert}
\newcommand{\norm}[1]{\Vert #1 \Vert}
\newcommand{\const}{{\rm const}}

\newcommand\sothat{{\rm :}\ }

\providecommand{\ltor}[1]{
\ifnum #1=1{\it i}\else\ifnum #1=2{\it ii}\else\ifnum #1=3{\it iii}
\else\ifnum #1=4 {\it iv}\fi\fi\fi\fi
}

\DeclareMathSymbol{\varPhi}{\mathord}{letters}{"08}
\DeclareMathSymbol{\varOmega}{\mathord}{letters}{"0A}

\font\thf cmssdc10 at 11pt

\theoremstyle{plain}
\newtheorem{theorem}{\thf Theorem}[section]

\newtheorem{lemma}[theorem]{\thf Lemma}
\newtheorem{corollary}[theorem]{\thf Corollary}
\newtheorem{proposition}[theorem]{\thf Proposition}

\theoremstyle{definition}
\newtheorem{definition}[theorem]{Definition}

\theoremstyle{remark}
\newtheorem{remark}[theorem]{Remark}

\makeatletter\@addtoreset{equation}{section}
\makeatother

\begin{document}

\title{
On Global attraction to solitary waves 
\\
for Klein-Gordon equation with concentrated nonlinearity
}

\author{
{\sc Elena Kopylova}
\footnote{Research supported by the Austrian Science Fund (FWF) under Grant No. P27492-N25 
and RFBR grant No. 16-01-00100
}
\\ 
{\it\small Faculty of Mathematics, Vienna University and IITP RAS
}}

\date{\version}

\maketitle

\begin{abstract}
The global attraction is proved for the nonlinear  
 3D Klein-Gordon equation with a 
nonlinearity concentrated at one point.
Our main result is the convergence  of each "finite energy solution"  to the manifold of 
all solitary waves  as $t\to\pm\infty$.
This global attraction is caused by the nonlinear energy transfer from lower harmonics 
to the continuous spectrum and subsequent dispersion radiation.

We justify this mechanism by the following strategy based on \emph{inflation of spectrum by the nonlinearity}.
We show that any {\it omega-limit trajectory} has the time-spectrum in the spectral gap $[-m,m]$ and
satisfies the original equation.  Then the application of the Titchmarsh Convolution Theorem reduces
the spectrum of each omega-limit trajectory to a single frequency  $\omega\in[-m,m]$.
\end{abstract}

\section{Introduction}
\label{int-results}
The paper concerns a nonlinear interaction of the Klein-Gordon field with a point oscillator.
The point interaction are  widely used in physical works. One of the well-known application
in dimension one is the Kronig-Penney model \cite{KrP}.
I n 3D case a rigorous mathematical definition of point interactions  was given
by Berezin and Faddeev \cite{BF}.
For the numerous literature concerning the  models with a  point interactions 
we refer to \cite{AGHH}.

In the case of  the  Schr\"odinger equations  the  nonlinear point interaction was justified
 in \cite{CFNT1,CFNT3} as a scaling limit of a regularized nonlinear Schr\"odinger dynamics.
We suppose that for the  Klein-Gordon equations   a justification can be done by suitable modification of methods 
\cite{CFNT1,CFNT3}, but it still remains an open question. 

 We consider the system  governed by the following equations
\begin{equation}\label{iKG}
\left\{\begin{array}{c}
\ddot \psi(x,t)=(\Delta-m^2)\psi(x,t)+\zeta(t)\delta(x)\\\\
\lim\limits_{x\to 0}(\psi(x,t)-\zeta(t)G(x))=F(\zeta(t))
\end{array}\right|\quad x\in\R^3,\quad t\in\R,\quad m>0,
\end{equation} 
where  $G(x)$ is the Green's function of operator $-\Delta+m^2$ in $\R^3$, i.e.
\begin{equation}\label{Green}
  G(x)=\frac{e^{-m|x|}}{4\pi|x|}.
\end{equation}
The nonlinearity  $F(\zeta)$ admits a real-valued potential
\begin{equation}\label{FU}
 F(\zeta)=\partial_{\ov\zeta}U(\zeta), \quad\zeta\in\C,\quad U\in C^2(\C),
\end{equation} 
 where  $\partial_{\ov\zeta}:=\frac 12(\pa_1+i\pa_2)$ with 
 $\zeta_1:=\Re\zeta$ and $\zeta_2:=\Im\zeta$.
 We assume that the potential $U(\zeta)$ is ${\rm U}(1)$-invariant, where ${\rm U}(1)$
 stands for the unitary group $e^{i\theta}$, $\theta\in\R~{\rm mod}~2\pi$.
  Namely, we assume that there exists $u\in C^2(\R)$ such that
\begin{equation}\label{U1}
U(\zeta)=u(|\zeta|^2),\quad\zeta\in\C.
\end{equation}
Conditions (\ref{FU}) and (\ref{U1}) imply that
\begin{equation}\label{F1}
F(\zeta)=b(|\zeta|^2)\zeta,\quad\zeta\in\C.
\end{equation}
where $b(\cdot)=u'(\cdot)\in C^1(\R)$ is real valued. Therefore
\begin{equation}\label{F11}
F(e^{i\theta}\zeta)=e^{i\theta}F(\zeta),\quad\theta\in\R,\quad\zeta\in\C.
\end{equation}
This symmetry implies that
$e^{i\theta}\psi(x,t)$ 
is a solution to the system (\ref{iKG}) if $\psi(x,t)$ is.
The system (\ref{iKG}) admits soliton solutions $\psi_{\omega}(x)e^{-i\omega t}$
with some $\omega\in (-m,m)$ and $\psi_\om\in L^2(\R^3)$.
Our main goal is the global attraction 
\[
\psi(x,t)\sim\psi_{\omega_{\pm}}(x)e^{-i\omega_{\pm} t},\quad t\to\pm\infty,
\]
for all  solutions from the Hilbert space ${\cal D}_F$ (see Definition \ref{def-space}),
where the asymptotics hold in local $L^2$-seminorms.

Similar global attraction was established  for the first time i) in \cite{Kom91}--\cite{KK2010a}
for 1D wave and 1D Klein-Gordon equations coupled to a nonlinear oscillator, 
ii) in \cite{KK2009, KK2010b} for t nD Klein-Gordon and Dirac equations with mean field  interaction, 
and iii) in \cite{C2013} for discrete in space and in time nD Klein-Gordon equation  equations  
interacting with a nonlinear oscillator.

In the context of the Schr\"odinger and wave equations
the point interaction of type (\ref{iKG})
was introduced in \cite{AAFT, AAFT1,  AGHH, KP, NP, NP1},
where the well-posedness of the Cauchy problem and the blow up solutions  were studied. 
In our recent paper \cite{K15} we proved the  well-posedness  for system (\ref{iKG}).

The asymptotic stability of solitary waves  have been obtained in \cite{BKKS,K09,KKS12}
for  1D Schr\"odinger equation coupled to nonlinear oscillator, in \cite{K10}
for  1D discrete Klein-Gordon equation coupled to nonlinear oscillator, and in \cite{ANO,ANO1}
for 3D Schr\"odinger equation with concentrated nonlinearity.

Global attraction to stationary state  for 3D wave equation  with the point interaction
has been proved for the first time in our recent paper \cite{K16}. 
However, the global attraction  for 3D Klein-Gordon equation equations 
with the point interaction has not been studied up to now.

Let us comment on our methods.
First,  we represent the solution  as a sum of  dispersive and  singular components.
The dispersive component is a solution to  the free Klein-Gordon equation, 
and the singular component is a solution to the coupled system of  the Klein-Gordon equation with delta-like sources 
and of  the first-order  nonlinear integro-differential equation which control the dynamics of the coefficients $\zeta(t)$
(see equation (\ref{delay})). The right  hand side of this equation is  the  value of the dispersive component 
at the singular point $x=0$.

The dispersive component  vanishes asymptotically in the  local seminorms 
and one remains with the contribution of the singular part only.
We show that the singular component  converges in the chosen topology to a solitary wave
which is a standing wave with a single frequency. 

Further, we  extract the omega-limit trajectories  of the singular component via the compactness argument.
Here the key role is played by  the absolute continuity of the spectral density 
$\tilde \zeta(\omega)$ outside the spectral gap.
The absolute continuity is a nonlinear version of Kato's theorem on the absence
of the embedded eigenvalues and provides the dispersion decay for the high energy
component. Any omega-limit trajectory  is the solution to (\ref{iKG}) with   
a function  $\eta(t)$ instead of $\zeta(t)$, 
which is a solution to a homogeneous nonlinear integro-differential equation
(\ref{zetalimeq}). The Fourier transform of $\eta(t)$ is a quasimeasure.
The theory of quasimeasures helps to prove the spectral inclusion (\ref{wFs}).

Finally, we apply the Titchmarsh convolution theorem (see  \cite[Theorem 4.3.3]{H})
to conclude that the support of the distributional Fourier transform
of each omega-limit trajectory is a singleton, i.e.
the spectrum of each omega-limit trajectory has a single frequency.
The Titchmarsh theorem controls the inflation of spectrum by the nonlinearity. 
Physically, these arguments justify the following binary mechanism of the
energy radiation, which is responsible for the attraction to the solitary waves: 
(i) the nonlinear energy transfer from the lower to higher harmonics, 
and (ii) the subsequent dispersion decay caused by the energy radiation to infinity.

The general scheme of the proof bring to mind the approach of \cite{KK2007,KK2010a}.
Nevertheless the Klein-Gordon equation  
with the point interaction requires new ideas due to a more singular character.
As a consequence, the formulation of the problem and the techniques used 
are not a straightforward generalization of the one-dimensional result \cite{KK2007}
and  the result \cite{KK2010a} for 3D equation with mean field interaction.

Our paper is organized as follows. In Section \ref{sect-results} we formulate the main theorem.
In Section \ref{sect-splitting} we separate the first dispersive component and study its decay properties.
In Section \ref{sect-bound} we construct spectral representation for the remaining singular component, 
and prove absolute continuity of its spectrum outside the spectral gap.
In section \ref{sect-comp} we establish compactness for the  singular component.
In Section \ref{sect-spectral} we study omega-limit trajectories of the solution.
In Section \ref{sect-proof}  we prove the main theorem and
in Appendix  we calculate some  Fourier transforms.
\section{Main results}
\label{sect-results}
\subsection*{Model}
We fix a nonlinear function $F:\C\to\C$ and define the domain
\begin{equation}\label{q}
D_F=\{\psi\in L^2(\R^3):\psi(x)=\psi_{reg}(x)+\zeta G(x),~~\psi_{reg}\in H^2(\R^3),~~
 \zeta\in\C,~~\psi_{reg}(0)=F(\zeta)\}
\end{equation}
which generally is not a linear space.
Let $H_F$ be a nonlinear operator on the domain $D_F$ defined by 
\begin{equation}\label{HF}
 H_F \psi=(\Delta-m^2)\psi_{reg},\quad\psi\in D_F.
\end{equation}
The system (\ref{iKG}) for $\psi(t)\in D_F$ reads
\begin{equation}\label{KG}
\ddot \psi(x,t)=H_F \psi(x,t),\quad x\in\R^3,\quad t\in\R.
\end{equation}
Let us introduce the phase space  for equation (\ref{KG}).
Denote the space
\begin{equation}\label{dD}
\dot D=\{\pi\in L^2(\R^3):\pi(x)=\pi_{reg}(x)+\eta G(x),
~~\pi_{reg}\in H^1(\R^3), ~~\eta\in\C\}
\end{equation}
Obviously, $D_F\subset\dot D$.
\begin{definition}\label{def-space}
\begin{enumerate}
\item
${\cal D}_F$ is the space of the states 
$\Psi=(\psi(x),\pi(x))\in D_F\oplus\dot D$ equipped with the finite norm
\begin{equation}\label{def-e}
\Vert\Psi\Vert_{{\cal D}_F}^2:=\norm{\psi_{reg}}_{H^2(\R^3)}^2+
\norm{\pi_{reg}}_{H^1(\R^3)}^2+|\zeta|^2+|\eta|^2.
\end{equation}
\item
${\cal X}$ is the Hilbert space of the states 
$\Psi=(\psi(x),\pi(x))\in H^2(\R^3)\oplus H^{1}(\R^3)$ equipped with the finite norm
\begin{equation}\label{def-e1}
\Vert\Psi\Vert_{{\cal X}}^2:=\norm{\psi}_{H^2(\R^3)}^2+\norm{\pi}_{H^{1}(\R^3)}^2.
\end{equation}
\end{enumerate}
\end{definition}
\begin{definition}
$H^{s}_{loc}=H^s_{loc}(\R^3)$,  $s=0,1,2,...$, denotes the Fr\'echet  space with finite seminorms
\begin{equation}\label{def-e-r1}
\norm{\psi }_{H^s_R}:=\norm{\psi}_{H^s(B_R)},\quad R>0,
\end{equation}
where $B_R$ is the ball of radius $R$.
\end{definition}
Denote $L^2_{loc}=H^{0}_{loc}$, ${\cal L}^2_{loc}=L^2_{loc}\oplus L^{2}_{loc}$ 
and ${\cal X}_{loc}=H^2_{loc}\oplus H^{1}_{loc}$. We set for $\Psi=(\psi,\pi)$
\[
\Vert\Psi\Vert_{{\cal L}^2_R}^2=\norm{\psi}_{L^2_R}^2+\norm{\pi}_{L^{2}_R}^2,\quad\quad
\Vert\Psi\Vert_{{\cal X}_R}^2=\norm{\psi}_{H^2_R}^2+\norm{\pi}_{H^{1}_R}^2,\quad R>0.
\]
\begin{remark}
The spaces ${\cal L}^2_{loc}$ are metrisable. 
The metrics can be defined by
\begin{equation}\label{mE}
{\rm dist}_{{\cal L}^2_{loc}}(\Psi_1,\Psi_2)=\sum\limits_{R=1}^{\infty}2^{-R}
\frac{\Vert\Psi_1-\Psi_2\Vert_{{\cal L}^2_R}}{1+\Vert \Psi_1-\Psi_2\Vert_{{\cal L}^2_R}}.
\end{equation}
\end{remark}
\subsection*{Global well-posedness}
For the global well-posedness, we assume that 
\begin{equation}\label{bound-below}
U(\zeta)\to\infty,\quad |\zeta|\to\infty.
\end{equation}
Denote $\Vert\cdot\Vert=\Vert\cdot\Vert_{L^2(\R^3)}$.
The next theorem is proved in \cite{K15}.
\begin{theorem}\label{theorem-well-posedness}
Let conditions (\ref{FU}), (\ref{U1}) and (\ref{bound-below}) hold.
Then 
\begin{enumerate}
\item
For every initial data $\Psi(0)=\Psi_0=(\psi_0,\pi_0)\in {\cal D}$  the Cauchy problem for
(\ref{KG}) has a unique solution $\psi(t)$ such that 
\[
\Psi(t)=(\psi(t),\dot\psi(t))\in C(\R,{\cal D}_F).
\]
\item
The energy is conserved:
\begin{equation}\label{ec}
{\cal H}(\Psi(t)):=
\frac 12 \Big(\Vert\dot\psi(t)\Vert^2+\Vert\nabla\psi_{reg}(t)\Vert^2
+m^2\Vert\psi_{reg}(t)\Vert^2\Big)+U(\zeta(t))=\const, \quad t\in\R.
\end{equation}
\item
The following a priori bound holds
\begin{equation}\label{apb}
|\zeta(t)|\le C(\Psi_0)\quad t\in\R. 
\end{equation}
\end{enumerate}
\end{theorem}
\subsection*{Solitary waves and the main theorem}
\begin{definition}\label{soldef}
(i) The solitary waves of equation  (\ref{KG}) are solutions of the form
\begin{equation}\label{sol1}
\psi(x,t)=e^{-i\omega t}\psi_\omega(x),\quad \omega\in\R,\quad \psi_\omega\in L^2(\R^3).
\end{equation}
(ii) The solitary manifold is the set
${\bf S}=\left\{\Psi_\omega=(\psi_\omega, -i\omega\psi_\omega)\sothat\omega\in\R\right\}$, where $\psi\sb\omega$
are the  amplitudes of solitary waves.
\end{definition}
The identity (\ref{F11}) implies that the set ${\bf S}$ is invariant under multiplication by 
$e\sp{i\theta}$, $\theta\in\R$.
Let us note that since $F(0)=0$ by (\ref{F1}), then for any $\omega\in\R$
there is a zero solitary wave with $\psi\sb\omega(x)\equiv 0$.
\begin{lemma}\label{sol-ex} (Existence of solitary waves).
Assume that $F(\zeta)$ satisfies (\ref{F1}). Then nonzero solitary waves may exist only 
for $\omega\in (-m,m)$. The amplitudes of solitary waves are given by 
\begin{equation}\label{psisol}
\psi_\omega(x)=q_\omega\frac{e^{-\sqrt{m^2-\omega^2}|x|}}{4\pi|x|}\in L^2(\R^3),\quad\omega\in (-m,m),
\end{equation}
where $q_\omega$ is the solution to 
\begin{equation}\label{qsol}
m-\sqrt{m^2-\omega^2}=4\pi b(|q_\omega|^2).
\end{equation}
\end{lemma}
\begin{proof}
We can split $\psi(x,t)=\psi_{reg}(x,t)+\zeta(t)G(x)$, where
\[
\psi_{reg}(x,t)=q_\omega e^{-i\omega t}\frac{e^{-\sqrt{m^2-\omega^2}|x|}-e^{-m|x|}}{4\pi|x|},
\quad\quad\zeta(t)=q_\omega e^{-i\omega t}.
\]
Evidently, $\psi_{reg}(\cdot,t)\in H^2(\R^3)$, $~~\zeta(\cdot)\in C_b(\R)$. 
Finally, the second equation of (\ref{iKG}) together with  (\ref{F1}) give
\[
q_\omega e^{-i\omega t}\frac{m-\sqrt{m^2-\omega^2}}{4\pi}
=q_\omega e^{-i\omega t} b(|q_\omega|^{2}).
\]
\end{proof}
At last, we assume that the nonlinearity is polynomial.
This assumption is crucial in our argument since
it will allow us to apply the Titchmarsh convolution theorem.
Now all our assumptions on $F$ can be summarized as follows.

\begin{equation}\label{f-is-such}
{\bf Assumption ~A}\qquad\qquad F(\zeta)=\partial_{\overline\zeta} U(\zeta),
\qquad
U(\zeta)=\sum\limits\sb{n=0}\sp{N}u_n\abs{\zeta}\sp{2n},\quad u_n\in\R,\quad u_N>0,\quad \ N\ge 2.\qquad\qquad
\end{equation}

In particular, this assumption guarantees that the nonlinearity
$F$ satisfies the bound (\ref{bound-below})
from Theorem~\ref{theorem-well-posedness}.
Our main result is the following theorem.
\begin{theorem}[Main Theorem]
\label{main-theorem}
Let Assumption~(\ref{f-is-such}) be satisfied.
Then for any $(\psi\sb 0,\pi\sb 0)\in {\cal D}_F$
the solution $\Psi(t)=(\psi(t),\dot\psi(t))$
to  {\rm (\ref{KG})} 
with $(\psi,\dot\psi)\at{t=0}=(\psi\sb 0,\pi\sb 0)$
converges to solitary manifold ${\bf S}$ in the space ${\cal L}^2_{loc}$:
\begin{equation}\label{cal-A}
\lim\sb{t\to\pm\infty}
{\rm dist}_{{\cal L}^2_{loc}}(\Psi(t),{\bf S})=0, 
\end{equation}
where ${\rm dist}_{{\cal L}^2_{loc}}(\cdot,\cdot)$ defined in (\ref{mE}).
\end{theorem}
It suffices to prove Theorem~\ref{main-theorem} for $t\to+\infty$.
We will only consider the solution $\psi(x,t)$ restricted to $t\ge 0$.

\section{Dispersive  component}
\label{sect-splitting}
Let $J_{1}$ be the Bessel function of order 1, and $\theta$ be the Heaviside function.
In \cite{K15} we proved that the solution $\psi(x,t)$ to (\ref{KG}) with initial data $\psi_0=\psi_{0,reg}+\zeta_0 G\in D_F$, 
$\pi_0=\pi_{0,reg}+\dot\zeta_0 G\in D$ is  given by
\begin{equation}\label{sol_sum}
\psi(x,t):= \psi_f(x,t)+\frac{\theta(t-|x|)}{4\pi|x|}\zeta(t-|x|)-\frac{m}{4\pi}
\int_0^t\frac{\theta(s-|x|)J_1(m\sqrt{s^2-|x|^2})}{\sqrt{s^2-|x|^2}}\zeta(t-s)ds, \quad t\ge 0.
\end{equation}
Here $\psi_f(x,t)\in C([0,\infty),L^2(\R^3))$ is a unique solution to the Cauchy problem for the free Klein-Gordon equation.
\begin{equation}\label{CP1}
\ddot{\psi}_f(x,t) = (\Delta-m^2)\psi_f(x,t),
\quad \psi_f(x,0) = \psi_0(x),\quad\dot\psi_f(x,0)  =  \pi_0(x),
\end{equation}
and $\zeta(t)\in C^1([0,\infty))$  is a unique solution to the Cauchy problem 
for the following first-order nonlinear integro-differential equation with delay
\begin{equation}\label{delay}
\frac {\dot\zeta(t)}{4\pi}-\frac {m}{4\pi}\zeta(t)+\frac {m}{4\pi}
\int_0^t\frac{J_1(ms)}{s}\zeta(t-s)ds+F(\zeta(t))=\lambda(t),\quad t\ge 0,
\quad \zeta(0)=\zeta_0,
\end{equation}
where $\lambda(t):=\lim\limits_{x\to 0}\psi_f(x,t)\in C([0,\infty))$.  
Note that the limit is well defined, $\lambda(t)$  is continuous for $t>0$, and it admits
a limit as $t\to +0$ (see  \cite{K15}).
The integral in (\ref{delay}) is bounded for all $t\ge 0$ due to well known properties 
of the Bessel function $J_1$: $J_1(r)\sim r^{-1/2}$ for  
$r\to\infty$, and $J_1(r)\sim r$ as $r\to 0$ (see for example \cite {O}).
Now we study the decay properties of the dispersive component  $\psi_f(x,t)$
for $t\to\infty$.
\begin{proposition}\label{decayM}
$\psi_{f}(x,t)$ decays in ${\cal X}_{loc}$ seminorms. That is,  $\forall R>0$
\begin{equation}\label{psi-f-decay}
\Norm{(\psi_{f}(t),\dot\psi_{f}(t)}_{{\cal X}_R}\to 0,\qquad t\to\infty.
\end{equation}
\end{proposition}
\begin{proof}
We split $\psi_f(x,t)$ as
\[
\psi_f(x,t)=\psi_{f,reg}(x,t)+\psi_{f,G}(x,t),\quad t\ge 0,
\]
where $\psi_{f,reg}$ and $\psi_{f,G}$  are defined
as solutions to the following Cauchy problems:
\begin{eqnarray}
&&
\ddot\psi\sb{f,reg}(x,t)=(\Delta-m^2)\psi\sb{f,reg}(x,t),
\qquad (\psi\sb{f,reg},\dot\psi\sb{f,reg})\at{t=0}=(\psi_{0,reg},\pi_{0,reg}),
\label{KG-cp-1}
\\
\nonumber
\\
&&
\ddot\psi\sb{f,G}(x,t)=(\Delta-m^2)\psi\sb{f,G}(x,t),
\qquad (\psi\sb{f,G},\dot\psi\sb{f,G})\at{t=0}=(\zeta_0G, \dot\zeta_0 G),
\label{KG-cp}
\end{eqnarray}
Since $(\psi_{0,reg},\pi_{0,reg})\in {\cal X}$, then  evidently,
\begin{equation}\label{psi-1-bounds}
(\psi_{f,reg},\dot\psi_{f,reg})\in C_{b}([0,\infty), {\cal X}).
\end{equation}
The following lemma states well known  decay in local seminorms for the free Klein-Gordon equation.
\begin{lemma}\label{lemma-decay-psi1} cf. \cite [Lemma 3.1]{KK2007})
Let $(u_0,v_0)\in {\cal X}$. Then $\forall R>0$
\begin{equation}\label{Uloc}
\Norm{{\cal U}(t)(u_0,v_0)}_{{\cal X}_R}\to 0,\qquad t\to\infty,
\end{equation}
where  ${\cal U}(t)$ is the dynamical group of the free Klein-Gordon equation.
\end{lemma}

Therefore, the first dispersive component  $\psi_{f,reg}(x,t)$ decays in ${\cal X}_{loc}$ seminorms. 
That is,  $\forall R>0$
\begin{equation}\label{Uloc1}
\Norm{(\psi_{f,reg}(\cdot,t),\dot\psi_{f,reg}(\cdot,t))}_{{\cal X}_R}\to 0,\qquad t\to\infty.
\end{equation}
Now we consider the second dispersive component $\psi_{f,G}$.
\begin{lemma}\label{lemma-decay-psi3}
$\psi_{f,G}(x,t)$ decays in ${\cal X}_{loc}$ seminorms. That is,  $\forall R>0$
\begin{equation}\label{psi2-decay}
\Norm{(\psi_{f,G}(t),\dot\psi_{f,G}(t)}_{{\cal X}_R}\to 0,\qquad t\to\infty.
\end{equation}
\end{lemma}
\begin{proof}
Let $\eta(x)$ be a smooth function with a support in $B_1$, such that
$\eta(x)=1$ for $x\in B_{1/2}$. We split G as
\[
G=\eta G+(1-\eta)G.
\]
 Lemma \ref{lemma-decay-psi1} implies that
\[
\Vert {\cal U}(t)\big(\zeta_0(1-\eta)G,\,\dot\zeta_0(1-\eta)G\big)\Vert_{{\cal X}_R}\to 0,
\qquad t\to\infty,\quad \forall R>0,
\]
since $\big(\zeta_0(1-\eta)G,\,\dot\zeta_0(1-\eta)G\big)\in {\cal X}$. Hence it suffices to prove that
\begin{equation}\label{u-decay}
\Norm{(u(t),\dot u(t))}_{{\cal X}_R}\to 0,\qquad t\to\infty,
\end{equation}
where
$(u(t),\dot u(t)):= {\cal U}(t)\big(\zeta_0\eta G,\,\dot\zeta_0\eta G\big)$.
The matrix kernel ${\cal U}(x-y,t)$ of the dynamical  group ${\cal U}(t)$
can be written as 
\begin{equation} \label{KGsol}
  {\cal U}(x-y,t)=\left( \begin{array}{ll}
  \dot U(x-y,t)    &              U(x-y,t)\\
  \ddot U(x-y,t)   &        \dot U(x-y,t)
  \end{array} \right),\quad x,y\in\R^3,\quad t>0,
\end{equation}
where
\begin{equation}\label{Gbess}
 U(z,t)=\frac{\delta (t-|z|)}{4\pi t}-
\frac{m}{4\pi}\frac{\theta(t-|z|)J_{1}
 (m\sqrt{t^{2}-|z|^{2}})}{\sqrt{t^{2}-|z|^{2}}},\quad z\in\R,\quad t>0.
\end{equation}
Well known asymptotics of the Bessel function imply that
\begin{equation}\label{W}
\Big|\partial_{t}^{k}\partial_{z}^{\beta}\frac{J_{1}
 (m\sqrt{t^{2}-|z|^{2}})}{\sqrt{t^{2}-|z|^{2}}}\Big|\le C(1+t)^{-3/2},\quad
t\ge 2|z|,\quad k=0,1,2,\quad |\beta|\le 2.
\end{equation}
Hence, for $|x|\le R$ and $t>2(R+1)$, we obtain
\[
|u(x,t)|+|\Delta u(x,t)|+|\dot u(x,t)|+|\nabla\dot u(x,t)|\le C t^{-3/2}.
\]
Then (\ref{u-decay}) follows.
\end{proof}
Finally, (\ref{Uloc}) and (\ref{psi2-decay}) imply (\ref{psi-f-decay}). 
\end{proof}
\begin{corollary}\label{cor-decay-psif}
From  (\ref{psi-f-decay}) immediately follows that
  \begin{equation}\label{psif-dec}
\lambda(t)=\psi_{f}(0,t)\to 0,\quad t \to\infty.
\end{equation}
\end{corollary}
In conclusion, let us show that
\begin{equation}\label{psi-2-bounds}
\psi_{f,G}(t)\in C_{b}([0,\infty), L^2(\R^3)).
\end{equation}
Indeed, the energy conservation for the free Klein-Gordon equation implies that
\[
{\cal U}(t)(0,G)=(U(t)G,\dot U(t)G)\in C_{b}([0,\infty), H^1(\R^3)\oplus L^2(\R^3)).
\]
Then
\[
\psi_{f,G}(t)=\zeta_0\dot U(t)G+\dot\zeta_0 U(t)G\in C_{b}([0,\infty), L^2(\R^3)).
\]
\section{Singular component}
\label{sect-bound}
\subsection*{Complex Fourier-Laplace transform}
In notation (\ref{sol_sum}) define the functions
\begin{equation}\label{pp-rep}
\psi_S(x,t):=\frac{\theta(t-|x|)}{4\pi|x|}\zeta(t-|x|)-\frac{m}{4\pi}
\int_0^t\frac{\theta(s-|x|)J_1(m\sqrt{s^2-|x|^2})}{\sqrt{s^2-|x|^2}}\zeta(t-s)ds\in C([0,\infty),L^2(\R^3)),
\quad t\ge 0.
\end{equation}
It is easy to verify that 
$\psi_S(x,t)$ is the solution to the Cauchy problem 
\begin{equation}\label{CP2}
\ddot\psi_S(x,t)= (\Delta-m^2)\psi_S(x,t) +\zeta(t)\delta(x),
\quad \psi_S(x,0) = 0,\quad\dot\psi_S(x,0)=0.
\end{equation}
The energy conservation (\ref{ec}) and a priory bound (\ref{apb})
imply that $\psi(t)\in C_{b}([0,\infty), L^2(\R^3))$.
Hence  (\ref{sol_sum}), (\ref{psi-1-bounds}) and (\ref{psi-2-bounds}) give that
\begin{equation}\label{psiS}
\psi_{S}(t)\in C_{b}([0,\infty), L^2(\R^3)).
\end{equation}
Let us analyze the  Fourier-Laplace transform of $\psi_{S}(x,t)$:
\begin{equation}\label{FL}
\tilde\psi_{S}(x,\omega)=\mathcal{F}_{t\to\omega}[\theta(t)\psi_{S}(x,t)]
:=\int_ 0^\infty e^{i\omega t}\psi_{S}(x,t)\,dt,\quad\omega\in\C^{+},\quad x\in\R^3,
\end{equation}
where $\C^{+}:=\{z\in\C:\;\Im z>0\}$. Note that
$\tilde\psi_{S}(\cdot,\omega)$ is an $L^2$-valued analytic function of $\omega\in\C^+$
due to (\ref{psiS}).
Equation (\ref{CP2})  implies that
\begin{equation}\label{FL-eq}
-\omega^2\tilde\psi_{S}(x,\omega)=(\Delta-m^2)\tilde\psi_S(x,\omega)
+\tilde\zeta(\omega)\delta(x),\quad \omega\in\C^{+},\quad x\in\R\sp 3,
\end{equation}
where $\tilde\zeta(\omega)$ is the Fourier-Laplace transform of $\zeta(t)$:
\begin{equation}
\tilde\zeta(\omega)={\cal F}_{t\to\omega}[\theta(t)\zeta(t)]=
\int_{0}^\infty e^{i\omega t}\zeta(t)\,dt.
\end{equation}

Applying the Fourier transform  to (\ref{FL-eq}), we get 
\begin{equation}\label{hat-tilde}
\hat{\tilde\psi}_{S}(\xi,\omega)
=\frac{\tilde\zeta(\omega)}{\xi^2+m^2-\omega^2},\qquad \xi\in\R^3,\qquad\omega\in\C^{+}.
\end{equation}
Denote 
\begin{equation}\label{def-k}
\varkappa(\omega)=\sqrt{\omega^2-m^2},
\qquad\Im \varkappa(\omega)>0,\qquad\omega\in\C^{+}.
\end{equation}
Then $\varkappa(\omega)$ is the analytic function on $\C\sp{+}$, and  $\tilde\psi_{S}(x,\omega)$
is given by
\begin{equation}\label{psi-s}
\tilde\psi_{S}(x,\omega)=\tilde\zeta(\omega)V(x,\omega), \quad V(x,\omega)=
\frac{e^{i\varkappa(\omega)|x|}}{4\pi|x|},\quad\quad\omega\in\C^{+}.
\end{equation}
We then have, formally, for any $\varepsilon>0$:
\begin{equation}
\psi_{S}(x,t)=\frac{1}{2\pi}\int_{\Im\omega=\varepsilon}e^{-i\omega t}\tilde\zeta(\omega)
V(x,\omega)\,d\omega
=\frac{1}{2\pi}\int_\R e^{-i\omega t}\tilde\zeta(\omega+i0)V(x,\omega+i0)\,d\omega
=\mathcal{F}\sb{\omega\to t}\sp{-1}
\big[\tilde\zeta(\omega)V(x,\omega)\big].
\end{equation}
\subsection*{Traces on the real line}
By (\ref{psiS}) the Fourier transform 
$\tilde\psi_{S}(\cdot,\omega)=\mathcal{F}_{t\to\omega}[\theta(t)\psi_{S}(\cdot,t)]$
is a tempered $L^2$-valued distribution of $\omega\in\R$.
It is the boundary value of the analytic function (\ref{FL})
 in the following sense:
\begin{equation}\label{bvp1}
\tilde\psi\sb{S}(\cdot,\omega)
=\lim\limits\sb{\varepsilon\to 0+}\tilde\psi\sb{S}(\cdot,\omega+i\varepsilon),\qquad\omega\in\R,
\end{equation}
where the convergence holds in $\mathscr{S}'(\R,L^2(\R^3))$.
Indeed,
\[
\tilde\psi_{S}(\cdot,\omega+i\varepsilon)
=\mathcal{F}_{t\to\omega}[\theta(t)\psi_{S}(\cdot,t)e\sp{-\varepsilon t}],
\]
while $\theta(t)\psi_{S}(\cdot,t)e^{-\varepsilon t}
\mathop{\longrightarrow}\limits_{\varepsilon\to 0+}\theta(t)\psi_{S}(\cdot,t)$
 in $\mathscr{S}'(\R,L^2(\R^3))$.
Therefore, (\ref{bvp1}) holds by the continuity of the Fourier transform
$\mathcal{F}_{t\to\omega}$ in $\mathscr{S}'(\R)$.

Similarly to (\ref{bvp1}), the distribution  $\tilde\zeta(\omega)$, $\omega\in\R$,
is the boundary values of the analytic in $\C^{+}$
function $\tilde\zeta(\omega)$, $\omega\in\C^{+}$:
\begin{equation}\label{bv}
\tilde\zeta(\omega)=\lim\limits_{\varepsilon\to 0+}
\tilde\zeta(\omega+i\varepsilon), \quad \omega\in\R,
\end{equation}
since the function $\theta(t)\zeta(t)$ is bounded.
The convergence holds in the space of tempered distributions $\mathscr{S}'(\R)$.

Let us justify that the representation (\ref{psi-s}) for $\tilde\psi_{S}(x,\omega)$
is also valid when $\omega\in\R\setminus\{-m;m\}$,  if the multiplication in (\ref{psi-s})
is understood in the sense of distribution. Namely,
\begin{lemma}\label{prop-uniform}
$V(x,\omega)$ is a smooth function of $\omega\in\R\setminus\{-m;m\}$ 
for any fixed $x\in\R^3\setminus \{0\}$, and the identity
\begin{equation}\label{p1r}
\tilde\psi_{S}(x,\omega)=\tilde\zeta(\omega)V(x,\omega),\quad\omega\in\R\setminus\{-m;m\}
\end{equation}
holds in the sense of distributions.
\end{lemma}
\begin{proof}
This lemma follows from (\ref{bvp1}) and (\ref{bv}) by the smoothness of $V(x,\omega)$ 
for $\omega\not =\pm m$.
\end{proof}
\subsection*{Absolutely continuous spectrum}
Note that $\R\setminus(-m,m)$ coincides with the continuous spectrum of the free Klein-Gordon 
equation, and the function $\omega \varkappa(\omega)$ is positive for $\omega\in\R\setminus[-m,m]$.
\begin{proposition}\label{tilde-f-plus-bounded} (cf. \cite[Proposition 2.3]{KK2009}
The distribution $\tilde \zeta(\omega+i0)$ is absolutely continuous for $\abs{\omega}>m$
and satisfies
\begin{equation}\label{rho-f-bound-we-have}
\int\sb{\abs{\omega}>m}\abs{\tilde \zeta(\omega)}^2\mathscr{M}(\omega)\,d\omega<\infty,
\quad{\rm where}~~~\mathscr{M}(\omega)=\frac{\varkappa(\omega)}{\omega}.
\end{equation}
\end{proposition}
\begin{remark}
Recall that $\tilde \zeta(\omega)$, $\omega\in\R$, is defined by (\ref{bv})
as the trace distribution: $\tilde \zeta(\omega)=\tilde \zeta(\omega+i0)$.
\end{remark}
\begin{proof}
For any $\delta>0$ denote $I_\delta=(-\infty,-m-\delta]\cup [m+\delta,\infty)$.
It suffices to prove that 
\begin{equation}\label{rho-f-bound-we-have-1}
\int_{I_\delta}\abs{\tilde \zeta(\omega)}^2\mathscr{M}(\omega)\,d\omega\le C,
\end{equation}
with some constant $C>0$ which does not depend on $\delta$.
First, the Parseval identity applied to
\[
\tilde\psi_S(x,\omega+i\epsilon)=\int_{0}^\infty\psi_S(x,t)e^{i\omega t-\epsilon t}\,dt,
\quad\epsilon>0,
\]
gives
\[
\int\sb{\R}\norm{\tilde\psi_S(\cdot,\omega+i\epsilon)}_{L^2(\R^3)}^2\,d\omega
=2\pi\int\sb{0}\sp\infty\norm{\psi_S(\cdot,t)}_{L^2(\R^3)}^2\, e^{-2\epsilon t}\,dt.
\]
Since $\sup\sb{t\ge 0}\norm{\psi_S(\cdot,t)}_{L^2(\R)}<\infty$
by (\ref{psiS}), we may bound the right-hand side
by $C\sb 1/\epsilon$, with some $C\sb 1>0$. Taking into account (\ref{psi-s}),
we arrive at the key inequality
\begin{equation}\label{e-plus-bounds}
\int\sb{\R}
|\tilde \zeta(\omega+i\epsilon)|^2\norm{V(\cdot,\omega+i\epsilon)}_{L^2(\R^3)}^2\,d\omega
\le\frac{C\sb 1}{\epsilon}.
\end{equation}
\begin{lemma}\label{lemma-h1-bounds}
There exists $n\in\N$ such that
for any $\delta>0$ and $0<\epsilon\le\delta/n$
\begin{equation}\label{n-half}
\norm{V(\cdot,\omega+i\epsilon)}_{L^2(\R^3)}^2\ge\frac{{\cal M}(\omega)}{16\pi\epsilon},
\qquad \omega\in I_\delta.
\end{equation}
\end{lemma}
\begin{proof}
The explicit formula (\ref{psi-s}) for $V(x,\omega+i\epsilon)$ implies
\begin{equation}\label{VV}
\norm{V(\cdot,\omega+i\epsilon)}_{L^2(\R^3)}^2=\frac{1}{4\pi}\int_0^{\infty}
|e^{i\varkappa(\omega+i\epsilon)r}|^2 dr=\frac{1}{8\pi\,\Im \varkappa(\omega+i\epsilon)}.
\end{equation}
Further, for $\omega\in I_\delta$  and $0<\epsilon\le \delta/n$ with sufficiently large $n\in\N$, we have
\begin{equation}\label{Imk}
\Im \varkappa(\omega+i\epsilon)=\Im\sqrt{(\omega+i\epsilon)^2-m^2}
=\varkappa(\omega)\Im\sqrt{1+(2i\epsilon\omega-\epsilon^2)/\varkappa^2(\omega)}
\le \frac{2\epsilon\omega}{\varkappa(\omega)}.
\end{equation}
Finally, (\ref{VV}) and (\ref{Imk}) imply  (\ref{n-half}). 
\end{proof}
Substituting (\ref{n-half}) into (\ref{e-plus-bounds}), we obtain the bound
\begin{equation}\label{fin}
\int_{I_\delta}\abs{\tilde\zeta(\omega+i\epsilon)}^2\mathscr{M}(\omega)\,d\omega
\le 16\pi C_1,\qquad 0<\epsilon\le\delta/n.
\end{equation}
We conclude that the set of functions
$g\sb{\delta,\epsilon}(\omega)=\tilde \zeta(\omega+i\epsilon)\sqrt{\mathscr{M}(\omega)},$
$0<\epsilon\le\epsilon(\delta)$ defined for $\omega\in I_\delta$,
is bounded in the Hilbert space $L^2(I_\delta)$, and, by the Banach Theorem, is weakly compact.
The convergence of the distributions (\ref{bv}) implies the weak convergence
$g_{\delta,\epsilon}\mathop{-\!\!\!\!-\!\!\!\!\rightharpoonup}\limits_{\epsilon\to 0+}g_{\delta}$
in the Hilbert space $L^2(I_\delta)$. The limit function $g_{\delta}(\omega)$ 
coincides with the distribution $\tilde\zeta(\omega)\sqrt{\mathscr{M}(\omega)}$
restricted onto $I_\delta$.
This proves the bound (\ref{rho-f-bound-we-have-1})
and finishes the proof of the proposition.
\end{proof}
\section{Compactness}
\label{sect-comp}
We are going to prove compactness of the set of translations of $\{\psi_S(x,t+s):~~s\ge 0\}$.
We start from the following lemma
\begin{lemma}\label{PC2}
For  any sequence $s_{j}\to\infty$  there exists an infinite subsequence 
(which we also denote by $s_{j}$) such that
\begin{equation}\label{zetalim}
\zeta(t+s_{j})\to\eta(t), \quad j\to\infty,\quad t\in\R
\end{equation}
for some  $\eta\in C_b(\R)$.
The  convergence is uniform on $[-T,T]$ for any $T>0$.
Moreover, $\eta(t)$ is the solution to
\begin{equation}\label{zetalimeq}
\frac {1}{4\pi}\dot\eta(t)-\frac {m}{4\pi}\eta(t)+\frac {m}{4\pi}
\int_0^\infty\frac{J_1(ms)}{s}\eta(t-s)ds+F(\eta(t))=0,\quad t\in\R.
\end{equation}
\end{lemma}
\begin{proof}
Theorem \ref{theorem-well-posedness}-iv), Corollary \ref{cor-decay-psif}
and equation (\ref{delay}) imply that $\zeta\in C^1_b(\R)$. 
Then (\ref{zetalim}) follows from the Arzel\'a-Ascoli theorem.
Further, for any $t\in\R$  we get
\begin{equation}\label{Intlim}
\int_0^{t+s_j}\frac{J_1(ms)}{s}\zeta(t+s_j-s)ds\to\int_{0}^{\infty}\frac{J_1(ms)}{s}\eta(t-s)ds,\quad j\to\infty 
\end{equation}
by the Lebesgue dominated convergence theorem.
Then   equation (\ref{delay}) for $\zeta(t)$ together with (\ref{psif-dec}) and (\ref{Intlim}) imply
(\ref{zetalimeq}).
\end{proof}
 Lemma \ref{PC2}  imply 
\begin{lemma}\label{PC}
The  convergences hold 
\begin{equation}\label{ePC}
\psi_S(\cdot,t+s_{j})\to\beta_S(\cdot,t)
:=\frac{\eta(t-|x|)}{4\pi|x|}-\frac{m}{4\pi}
\int_0^\infty\frac{\theta(s-|x|)J_1(m\sqrt{s^2-|x|^2})}{\sqrt{s^2-|x|^2}}\eta(t-s)ds, 
\quad j\to\infty,\quad t\in\R,
\end{equation}
\begin{equation}\label{ePC1}
\dot\psi_S(\cdot,t+s_{j})\to\dot\beta_S(\cdot,t)
=\frac{\dot\eta(t-|x|)}{4\pi|x|}-\frac{m}{4\pi}
\int_0^\infty\frac{\theta(s-|x|)J_1(m\sqrt{s^2-|x|^2})}{\sqrt{s^2-|x|^2}}\dot\eta(t-s)ds, 
\quad j\to\infty,\quad t\in\R,
\end{equation}
in the topology of $C_b([-T,T],L^2_{loc})$ for any $T>0$. 
\end{lemma}
\begin{proof}
The convergence (\ref{ePC})  follows immediately from
(\ref{pp-rep}), (\ref{zetalim}) and the Lebesgue dominated convergence theorem.
Let us  prove the convergence of $\dot\psi_S(\cdot,t+s_j)$.
Equations (\ref{delay}) and (\ref{zetalimeq}) imply that
\begin{equation}\label{dotzeta}
\dot\zeta(t+s_j)\to \dot\eta(t),\quad j\to\infty,
\end{equation}
uniformly on $[-T,T]$ for any $T>0$. 
Further, differentiating (\ref{pp-rep})  for $t>|x|$ gives
\[
\dot\psi_S(x,t)=\frac{\dot\zeta(t-|x|)}{4\pi |x|}
-\frac{m}{4\pi}\frac{J_1(m\sqrt{t^2-|x|^2})}{\sqrt{t^2-|x|^2}}\zeta(0)
-\frac{m}{4\pi}
\int_0^t\frac{\theta(s-|x|)J_1(m\sqrt{s^2-|x|^2})}{\sqrt{s^2-|x|^2}}\dot\zeta(t-s)ds,
\]
which  imply (\ref{betaL2})   by (\ref{dotzeta}).
\end{proof}
\begin{remark}
From  (\ref{psiS}) it follows that
\begin{equation}\label{betaL2}
\beta_S(\cdot,t)\in L^{\infty}(\R,L^2(\R^3)).
\end{equation}
\end{remark}
\section{Nonlinear spectral analysis}
\label{sect-spectral}
We call   {\it an omega-limit trajectory} any function $\beta_S(x,t)$ that can appear as a limit in
(\ref{ePC}). Proposition \ref{decayM} demonstrates that the long-time asymptotics of the solution 
$\psi(x,t)$ in $L^2_{loc}$ depends only on the singular component $\psi_S(x,t)$.
Namely, the convergences (\ref{ePC}),  and system (\ref{iKG}) together with
(\ref{sol_sum}), (\ref{psi-f-decay}) and (\ref{psif-dec})
imply that any $\beta_S(x,t)$ is a solution to  (\ref{iKG})  with $\eta(t)$ instead $\zeta(t)$:
\[
\left\{\begin{array}{c}
\ddot\beta_S(x,t)=(\Delta-m^2)\beta_S(x,t)+\eta(t)\delta(x)\\\\
\lim\limits_{x\to 0}(\beta_S(x,t)-\eta(t)G(x))=F(\eta(t))
\end{array}\right|\quad t\in\R.
\]
In this section we prove the following proposition.
\begin{proposition}\label{olt}
Every omega-limit trajectory is a solitary wave, that is,
\begin{equation}\label{betasol}
\beta_S(x,t)=\psi_{\omega_+}(x)e^{-i\omega_+t}, \quad x\in\R^3, \quad t\in\R,
\end{equation}
with some $\omega_+\in\R$. 
\end{proposition}
\subsection{Reduction of  spectrum}
\begin{lemma}\label{prop-beta}
$\supp\tilde\eta \subset [-m,m]$.
\end{lemma}
\begin{proof}
Due to (\ref{zetalim}) and the continuity of the Fourier transform in $\mathscr{S}'(\R)$, we have
\[
\alpha(\omega)\tilde\zeta(\omega)e^{-i\omega s_l}\mathop{\stackrel{\mathscr{S}'}\longrightarrow}
\alpha(\omega)\tilde\eta(\omega),\quad j\to\infty. 
\]
for any $\alpha\in C_0^{\infty}(\R)$ such that $\supp\alpha\cap[-m,m]=\emptyset$.
The products $\alpha(\omega)\tilde\zeta(\omega)$ are  absolutely continuous measures since
$\tilde\zeta(\omega)$ is locally $L^2$ for $\omega\in\R\setminus [-m,m]$ 
by Proposition \ref{tilde-f-plus-bounded}.
Then $\tilde \eta(\omega)=0$   for $\omega\notin [-m,m]$ by the Riemann-Lebesgue Theorem.
\end{proof}

Using (\ref{p1r}) and taking into account that $V(x,\omega)$ is smooth for $\omega\not=\pm m$ and $x\not =0$,
we obtain the following relation, which holds in the sense of distributions:
\begin{equation}\label{ber}
\tilde\beta_S(x,\omega)=\tilde\eta(\omega)V(x,\omega),
\qquad\omega\in\R \setminus\{\pm m\}.
\end{equation}
Since $V(x,\omega)\not =0$ for $\omega\in \R$ it follows from Lemma \ref{prop-beta} that
\begin{equation}\label{supptet}
\supp\tilde\beta_S(x,\cdot)\subset[-m,m]. 
\end{equation}
\subsection{Spectral inclusion and the Titchmarsh  theorem}
We will derive (\ref{betasol}) from the following identity
\begin{equation}\label{zetabeta}
\eta(t)=Ce^{-i\omega_+t},\quad t\in R,\quad \omega_+\in [-m,m],
\end{equation}
which will be proven in three steps. We start with an investigation of $\supp\tilde \eta$.
\begin{lemma}\label{prop-b-g}
The following spectral inclusion holds:
\begin{equation}\label{wFs}
\supp \widetilde{F(\eta)}\subset\supp\tilde \eta.
\end{equation}
\end{lemma}
\begin{proof}
Applying the Fourier  transform to (\ref{zetalimeq}), we get by the theory of
quasimeasures (see \cite{KK2007}) that
\begin{equation}\label{TK}
\widetilde{F(\eta)}(\omega)=\frac{1}{4\pi}(i\omega+m-m\tilde K(\omega))\tilde\eta(\omega)
=\frac{1}{4\pi}(m-\sqrt{m^2-\omega^2})\tilde\eta(\omega),\quad |\omega|\le m,
\end{equation}
where $\tilde K(\omega)=\frac 1m\big(\sqrt{m^2-\omega^2} +i\omega\big)$ 
is the Fourier transform  of the function $K(t)=\theta(t)J_1(mt)/t\in L^1(\R)$
(see Appendix),
and $\tilde\eta(\omega)$ is a quasimeasure.
Then (\ref{wFs}) follows.
\end{proof}
The second step is the following lemma
\begin{lemma}
For any omega-limit trajectory
\begin{equation}\label{zeta_const}
|\eta(t)|=\const,\quad t\in\R.
\end{equation}
\end{lemma}
\begin{proof}
Our main assumption (\ref{f-is-such}) implies that the function $F(\eta(t))$ admits the representation
\begin{equation}\label{Fae}
  F(\eta(t))=a_\eta(t)\eta(t),
\end{equation}
where, according to (\ref{f-is-such})
\begin{equation}\label{a_beta}
  a_\eta(t)=\sum\limits_{n=1}^{N} 2 n u_{n}|\eta(t)|^{2n-2}.
\end{equation}
Both function $\eta(t)$ and $a_\eta(t)$ are bounded continuous functions in $\R$ by Lemma \ref{PC2}. 
Hence, $\eta(t)$ and $a_\eta(t)$ are tempered distributions. According to (\ref{supptet})
$\supp\tilde\eta\subset [-m,m]$, $\supp\tilde{\overline\eta}\subset  [-m,m]$,
and then $\tilde a_\eta$ also has a bounded support.
Denote ${\bf F}=\supp \tilde F(\eta)$, ${\bf A}=\supp \tilde a_\eta$,  ${\bf Z}=\supp \tilde\eta$.
Then the spectral inclusion (\ref{wFs}) gives
\begin{equation}\label{wFs1}
  {\bf F}\subset{\bf Z}.
\end{equation}
On the other hand,  applying the Titchmarsh convolution theorem (see \cite [Theorem 4.3.3]{H}) to (\ref{Fae}),
we obtain
\begin{equation}\label{Tit}
  \inf{\bf F}=\inf{\bf A}+\inf{\bf Z},\quad \sup{\bf F}=\sup{\bf A}+\sup{\bf Z}.
\end{equation}
From (\ref{wFs1}) and (\ref{Tit}) it follows that $\inf{\bf A}=\sup{\bf A}=0$, and hence ${\bf A}\subset \{0\}$.
Thus, we conclude that $\supp\tilde a_\eta={\bf A}\subset \{0\}$, and therefore
the distribution $\tilde a_\eta(\omega)$ is a finite linear combination of $\delta(\omega)$
and its derivatives. Then $a_\eta(t)$ is a polynomial in $t$; since $a_\eta(t)$ is bounded 
by Lemma \ref{PC2}, we conclude that $a_\eta(t)=\const$.
Finally, (\ref{zeta_const}) follows since $a_\eta(t)$ is a polynomial in $|\eta(t)|$,  
and its degree $2N-2\ge 2$ by (\ref{f-is-such}) and (\ref{a_beta}). 
\end{proof}
Now (\ref{zeta_const}) means that $\eta(t)\overline\eta(t)\equiv C=\const$, and  then 
$\tilde \eta*\tilde{\overline\eta}=2\pi C\delta(\omega)$. Hence, if $\eta$ is 
not identically zero,
the  Titchmarsh theorem implies that ${\bf Z}=\omega_+\in [-m,m]$. Indeed, 
\[
0=\sup{\bf Z}+\sup(-{\bf Z})=\sup{\bf Z}-\inf{\bf Z},
\]
and hence $\inf{\bf Z}=\sup{\bf Z}$. Therefore, $\tilde \eta$ is a finite linear 
combination of $\delta(\omega-\omega_+)$ and its derivatives. But the derivatives could
not be present because of the boundedness of $\eta(t)$.
Thus $\tilde \eta\sim\delta(\omega-\omega_+)$, which implies (\ref{zetabeta}).
\medskip\\
{\bf Proof of Proposition~\ref{olt}}.
Substituting (\ref{zetabeta}) in the RHS of (\ref{ePC}), we obtain
\begin{eqnarray}\nonumber
\beta_S(x,t)&=&\frac{Ce^{-i\omega_+(t-|x|)}}{4\pi|x|}
-\frac{mC}{4\pi}\int_0^\infty\frac{\theta(s-|x|)J_1(m\sqrt{s^2-|x|^2})}{\sqrt{s^2-|x|^2}}
e^{-i\omega_+(t-s)}ds\\
\label{betam}
&=&\frac{Ce^{-i\omega_+t}}{4\pi}
\Big(\frac{e^{i\omega_+|x|}}{|x|}-m\tilde L(x,\omega_+)\Big)=
\frac{Ce^{-\sqrt{m^2-\omega_+^2}|x|}}{4\pi|x|}e^{-i\omega_+ t}.
\end{eqnarray}
Here 
$\tilde L(x,\omega_+)=\frac{1}{|x|m}\big(e^{i|x|\omega_+}-e^{i|x|\sqrt{\omega_+^2-m^2}}\big)$ is the Fourier transform of the function
$L(x,t)=\frac{\theta(t-|x|)J_1(m\sqrt{t^2-|x|^2})}{\sqrt{t^2-|x|^2}}$ 
(see Appendix). Hence, (\ref{betasol}) holds and $\beta_S(x,t)$ is a solitary wave.
\hfill$\Box$
\begin{remark}
If (\ref{zetabeta}) holds with some $|\omega_+|<m$, then
 (\ref{betasol}) follows immediately from (\ref{ber}).
\end{remark}
\section{ Proof of Theorem~\ref{main-theorem}}
\label{sect-proof}
Due to Proposition \ref{decayM}  it suffices to prove  that
\begin{equation}\label{cal-B}
\lim_{t\to\infty}
{\rm dist}_{{\cal L}^2_{loc}}(\Psi_S(t),{\bf S})=0, 
\end{equation}
where $\Psi_S(t)=(\psi_S(t),\dot\psi_S(t))$.
Assume by contradiction that there exists a sequence $s_j\to \infty$ such that
\begin{equation}\label{cal-C}
{\rm dist}_{{\cal L}^2_{loc}}(\Psi_S(s_j),{\bf S})\ge\delta,\quad\forall j
\end{equation}
for some $\delta>0$.  According to Lemmas \ref{PC2} and \ref{PC}, 
and formula (\ref{betasol}) 
there exist a subsequence $s_{j_k}$ of the sequence $s_j$ and an amplitude $\psi_{\omega_+}$
such that the following convergences  hold
\[
\Psi_S(t+s_{j_k})\to (\psi_{\omega_+}e^{-i\omega_+ t},-i\omega_+\psi_{\omega_+}e^{-i\omega_+ t}),
\quad j_k\to\infty,\quad t\in\R.
\]
This implies that $\Psi_S(s_{j_k})\to(\psi_{\omega_+},-i\omega\psi_{\omega_+})$, which
contradict (\ref{cal-C}).
This completes the proof of Theorem~\ref{main-theorem}.
\hfill$\Box$
\appendix
\section{Appendix. Fourier transforms}
\label{A}
Here we calculate the Fourier transforms  of   
$L(x,t)=\theta(t-|x|)J_1(m\sqrt{t^2-|x|^2})/\sqrt{t^2-|x|^2}$ and  $K(t)=\theta(t)J_1(mt)/t=L(0,t)$, which we have used
in (\ref{TK}) and (\ref{betam}).
 Recall, that the function
\begin{equation}\label{U}
U(x,t)=\frac{\delta(t-|x|)}{4\pi|x|}-\frac{m}{4\pi}L(x,t),\quad x\in\R^3, \quad t\in\R
\end{equation}
is the fundamental solution to the Klein-Gordon equation:
\[
\ddot U(x,t)=(\Delta-m^2)U(x,t)+\delta(x)\delta(t),\quad x\in\R^3, \quad t\in\R.
\]
Applying the Fourier transforms in $t$, we obtain
\[
(\Delta-m^2+\omega^2)\tilde U(x,\omega)=-\delta(x),\quad x\in\R^3, \quad \omega\in\R.
\]
Note that $\tilde U(\cdot,\omega)$ 
is an analytic and bounded function of $\om\in\C^+$ with values 
in tempered distributions on $\R^3$.
Moreover, $\tilde U(\cdot,\omega)$ is a radial distribution, and hence, 
it coincides with $V(\cdot,\omega)$ by 
 (\ref{def-k})--(\ref{psi-s}):
\[
\tilde U(x,\omega)=V(x,\omega)=\frac{e^{i\varkappa(\omega)|x|}}{4\pi|x|},
\quad\omega\in\overline \C^{+}.
\]
Therefore,
\begin{equation}\label{tL}
\tilde L(x,\omega)=\frac{4\pi}{m}\big(\frac{e^{i\omega |x|}}{4\pi|x|}-V(x,\omega)\big)
=\frac{1}{m|x|}\big(e^{i\omega |x|}-e^{-\sqrt{m^2-\omega^2}|x|}\big), \quad |\omega|\le m.
\end{equation}
Passing to the limit,
we obtain
\begin{equation}\label{Pieq}
\tilde  K(\omega)=\lim\limits_{x\to 0}\tilde L(x,m)=\frac 1m\big(\sqrt{m^2-\omega^2}+i\omega),\quad |\omega|\le m,
\end{equation}
\begin{remark}
Formula (\ref{Pieq}) agrees  with \cite [ Sections 1.12(4) and  2.12 (5)]{BE},
which are cosine and sine transforms of $K(t)$.
\end{remark}

\end{document}